\documentclass[12pt]{amsart}

\usepackage{amsmath,amssymb,amsthm}

\renewcommand{\title}[1]{%
  \centerline{\textbf{#1}}
  \vspace{36pt}}
\renewcommand{\date}[1]{%
\vspace{-28pt}
\centerline{\small #1}
\vspace{28pt}}

\newtheorem*{conjecture*}{Conjecture}

\newtheorem*{corollary*}{Corollary}

\newtheorem*{frenchcorollary*}{Corollaire}

\newtheorem*{theorem*}{Theorem}
\newtheorem*{frenchtheorem*}{Th\'eor\`eme}
\theoremstyle{definition}

\begin{document}

\begin{center}
\bfseries 
An oft cited letter from Tate to Serre on computing local
heights on elliptic curves
\end{center}

\noindent
This document contains a letter written by John Tate to Jean-Pierre
Serre in~1979 in which Professor Tate describes a fast and elegant algorithm for
computing local canonical heights on elliptic curves over local
fields~$K_v$ other than~$K_v=\mathbb{C}$. Accounts of Tate's algorithm
were published in~[1] and~[2], the former including a modified
algorithm that also allows $K_v=\mathbb{C}$. But over the years there
have been numerous articles that make direct reference to the original
handwritten letter, so Professor Tate has kindly given me permission
to post this typeset version on the ArXiv.
\begin{flushright}
\begin{tabular}{l}
Joseph H. Silverman\\July 24, 2012\\
\end{tabular}
\end{flushright}

{\par\noindent\hangindent1.5em \hangafter1
\hbox to1.5em{[1]\hfil\enspace}%
J.~H. Silverman, Computing heights on elliptic curves,
\emph{Math. Comp.} \textbf{51} (1988), 339--358.
\par\noindent\hangindent2em \hangafter1
\hbox to1.5em{[2]\hfil\enspace}%
H. Tsch{\"o}pe and H. Zimmer, 
Computation of the {N}\'eron-{T}ate height on elliptic curves,
\emph{Math. Comp.} \textbf{48} (1987), 351--370.\par}

\vspace{12pt}
\hbox to\hsize{\hfil\hbox to.75\hsize{\hrulefill}\hfil}
\vspace{8pt}

\begin{flushright}
October 1, 1979
\end{flushright}

\noindent
Dear Serre,

\vspace{12pt}
\noindent
Here is an appendix to my letter of 21 June 1968. I want to explain what seems
to me an efficient way to compute $\lambda_\infty$, the ``canonical
quasifunction'' on an elliptic curve over $\mathbb{R}$. The method can be used
on an HP25 (see enclosed programs). On the other hand, from a theoretical
point of view it can be used to show the existence of the $\lambda_v$ on any
local field $K_v$ \emph{except} in case $K_v \cong \mathbb{C}$. The trick is to
use as parameter a function $t = \frac{1}{x+r}$ with a
double zero at 0, but with $r$ chosen so that the 2 poles of $t$ are conjugate
quadratic over $K_v$ rather than being in $K_v$, so that \emph{$t$ is bounded}
on the group $E(K_v)$ of rational points, and the computer can compute
$t(2^n P)$ by iteration, \emph{never overflowing.} The method is summed up in:
\begin{theorem*}
Let $K$ be a local field. Let $E$ be an elliptic curve over $K$ defined by the
usual equation 
\[
  y^2 + a_1xy + a_3y = x^3 + a_2x^2 + a_4x +a_6
\]
and define $b_2$, $b_4$, $b_6$, and $b_8$ as usual. Suppose there is
an open subgroup $\Gamma$ of $E(K)$ such that 
\underline{\smash{$x(P)$ does not take the value $0$ for $P\in \Gamma$}.}
\par
\textup{[Examples:} $(K = \mathbb{R})$ If
  $b_6 < 0$, we can take $\Gamma = E(\mathbb{R})$; if $\Delta > 0$, and
  $b_2 = 0$ \textup{(}or more generally if the point of order~2 on the
  connected component of $0$ has a strictly positive
  $x$-coordinate\textup{)} then we can take $\Gamma =
  E^0(\mathbb{R})$, the connected component of $0$ in
  $E(\mathbb{R})$.\textup{]}
\par
Suppose $P \in \Gamma$. Define sequences $(t_n)$, $(W_n)$, and $(Z_n)$,
$n\geq 0$, inductively by:
\[
  t_0 = \frac{1}{x(P)} \text{ and }t_{n+1} = \frac{W_n}{Z_n}, 
  \text{ where }
  \left\{
    \begin{aligned}
    W_n &= 4t_n +b_2t^2_n +2b_4t^3_n + b_6t^4_n, \\
    Z_n &= 1-b_4 t^2_n -2b_6t^3_n -b_8 t^4_n.\\
    \end{aligned}
  \right.
\]
There are constants $\varepsilon > 0$ and $M$, independent of $P$, such that
\[
  \text{$|t_n| \leq M$\quad and\quad $|Z_n| \geq \varepsilon$\quad for all $n$,}
\]
and consequently there is a bound $N$ such that $\log |Z_n|\leq N$ for
all $n$ and all $P\in \Gamma$.  Put
\[
  \mu(P) = \sum^\infty_{n=0} \frac{1}{4^n} \log |Z_n|
  \quad\text{and}\quad
  \lambda(P) = \frac{1}{2} \log \bigl|x(P)\bigr|+ \frac{1}{8} \mu(P).
\]
Then $\lambda$ is a function continuous on $\Gamma - \{0\}$, bounded outside
of any neighborhood of $0$ in $\Gamma$, such that
\[
  \lambda(P) -\frac{1}{2}\log \bigl|x(P)\bigr| \to 0\quad\text{as}\quad P\to 0,
\]
and such that
\[
  \lambda(2P) = 4\lambda(P)-\log\bigl|Y(P)\bigr|,
  \quad\text{if $2P \neq 0$, where $Y=2y+a_1x + a_3$.}
\]
\end{theorem*}
\begin{proof}
We have $Y^2 = 4x^3 + b_2x^2 + 2b_4x + b_6$. Let $x_2 = x(2P)$.  Then
\begin{equation}
  x_2 = \frac{x^4 -b_4x^2 -2b_6x -b_8}{4x^3 +b_2x^2 +2b_4x +b_6} =
  \frac{x^4Z}{Y^2} ,
  \tag{$*$}
  \label{eqn1}
\end{equation}
where 
\[
  Z=1 -b_4x^{-2} -2b_6x^{-3} -b_8x^{-4}.
\]
Now, under our underlined hypothesis, there is an $\varepsilon > 0$
such that
\[
  \text{$\bigl|x(P)\bigr|\geq \varepsilon$\quad for\quad $P \in \Gamma$.}
\]
Hence the function
$x^{-1}$ is bounded on $\Gamma$, and so $Z$ is also. But $Z$ is
bounded away from $0$, too, on $\Gamma$, because $|x_2|\geq \varepsilon$,
and $Z$ has no zero in common with $x^{-1}$ and $Y$.  Thus $\log |Z|$
is bounded on $\Gamma$. From \eqref{eqn1} we see that, in the notation
of the theorem, we have $Z_n = Z(2^n P)$, and consequently
\[
  \mu(P) = \sum^\infty_{n=0} \frac{1}{4^n} \log \bigl|Z(2^n P)\bigr|
\]
is well-defined, continuous, and bounded on $\Gamma$; and satisfies
\begin{equation}
  \mu(P) = \log \bigl|Z(P)\bigr| + \frac{1}{4}\mu(2P).
  \tag{$**$}
  \label{eqn2}
\end{equation}
From \eqref{eqn1} we find for $P$ such that $2P \neq 0$,
\[
  \log \bigl|x(2P)\bigr| 
   = 4\log\bigl|x(P)\bigr|+\log\bigl|Z(P)\bigr| - 2\log\bigl|Y(P)\bigr|.
\]
Eliminating $\log \bigl|Z(P)\bigr|$ from these last two equations we
find that the function
\[
  \lambda = \frac{1}{2} \log|x|+\frac{1}{8}\mu
\]
 satisfies
\[
  \lambda(2P) = 4\lambda(P) - \log \bigl|Y(P)\bigr|
  \quad\text{for $P \in \Gamma$, $2P \neq 0$.}
\]
Moreover,
\[
  \lambda(P) - \frac{1}{2}\log \bigl|x(P)\bigr| = \frac{1}{8}\mu(P) \to 0\quad 
  \text{as}\quad P \to 0,
\]
because $\mu$ is continuous and $\mu(0) = 0$, as one sees by putting $P=0$
in \eqref{eqn2} and using $Z(0)=1$.
\end{proof}

\vspace{12pt}
\noindent
Salut et Fraternit\'e

\vspace{12pt}
\noindent
J. Tate

\end{document}